\documentclass[11pt]{amsart}
\usepackage{array,amsmath, enumerate,  url, psfrag}
\usepackage{amssymb, amsaddr}
\usepackage{graphicx,subfigure}
\usepackage{color}

\makeatletter
 \usepackage{fullpage} 
 \usepackage{amsthm}

\usepackage{fullpage}

\makeatother

\begin{document}
\newcommand{\n}{\noindent}
\newtheorem{thm}{Theorem}[section]
\newtheorem{conj}[thm]{Conjecture}
\newtheorem{Lemma}[thm]{Lemma}
\newtheorem{res}[thm]{Result}
\newtheorem{alg}[thm]{Algorithm}
\newtheorem{defn}[thm]{Definition}
\newtheorem{cor}[thm]{Corollary}
\newtheorem{observ}[thm]{Observation}
\newtheorem{ques}[thm]{Question}
\newtheorem{prop}[thm]{Proposition}
\newcommand{\red}{\textcolor{red}}
\newcommand{\di}{\displaystyle}
\def\dfc{\mathrm{def}}
\def\F{{\mathcal F}}
\def\cH{{\cal H}}
\def\cT{{\cal T}}
\def\C{{\mathcal C}}
\def\cA{{\cal A}}
\def\cB{{\mathcal B}}
\def\P{{\mathcal P}}
\def\Q{{\mathcal Q}}
\def\Z{{\mathcal Z}}
\def\cP{{\mathcal P}}
\def\cp{\alpha'}
\def\Frk{F_k^{2r+1}}

\title{A relaxation of the strong Bordeaux Conjecture}

\author{Ziwen Huang \hskip 0.5in Xiangwen Li$^{\dagger}$}
\address{\small Department of Mathematics, Huazhong Normal University, Wuhan, 430079, China}
\email{xwli68@mail.ccnu.edu.cn}
\thanks{$^{\dagger}$The research was partially supported by the  Natural Science Foundation of China (11171129)  and by Doctoral Fund of Ministry of Education of China (20130144110001)}

\author{Gexin Yu}
\address{\small Department of Mathematics, The College of William and Mary, Williamsburg, VA, 23185, USA.\\Department of Mathematics, Huazhong Normal University, Wuhan, 430079, China}
\email{gyu@wm.edu}




\date{\today}

\begin{abstract}
Let $c_1, c_2, \cdots, c_k$ be $k$ non-negative integers. A graph $G$ is $(c_1, c_2, \cdots, c_k)$-colorable  if the vertex set can be partitioned into $k$ sets $V_1,V_2, \ldots, V_k$,  such that the subgraph $G[V_i]$, induced by $V_i$, has maximum degree at most $c_i$ for $i=1, 2, \ldots, k$. Let $\mathcal{F}$ denote the family of plane graphs with neither adjacent 3-cycles nor $5$-cycle.  Borodin and Raspaud (2003) conjectured that each graph in $\mathcal{F}$ is $(0,0,0)$-colorable.  In this paper, we prove that each graph in $\mathcal{F}$ is $(1, 1, 0)$-colorable, which improves the results by Xu (2009) and Liu-Li-Yu (2014+).\\

\begin{flushleft}
\end{flushleft}
\end{abstract}

\maketitle

\section{Introduction}
All graphs considered in this paper are finite, simple, and undirected.  Call a graph $G$ {\em planar} if it can be embedded into the plane so that  its edges meet only at their ends. As proved by Graey et al~\cite{Gar}, the problem of deciding whether a planar graph is properly 3-colorable is NP-complete. In 1959, Gr\"{o}tzsch~\cite{G59} showed  that every triangle-free planar graph is 3-colorable. A lot of research was devoted to find sufficient conditions for a planar graph to be $3$-colorable, by allowing a triangle together with some other conditions. The well-known Steinberg's conjecture~\cite{S76}  stated below is one of such numerous efforts.

\begin{conj}[Steinberg, \cite{S76}] \label{con1}
All planar graphs without $4$-cycles and $5$-cycles are $3$-colorable.
\end{conj}

Towards this conjecture, Erd\H{o}s suggested to find a  constant $c$ such that a planar graph without cycles of length from $4$ to $c$ is $3$-colorable.  The best constant people so far is $c=7$, found by Borodin, Glebov, Raspaud, and Salavatipour~\cite{BGRS05}. For more results,  see the recent nice survey by Borodin~\cite{B12}.

Another relaxation of the conjecture is to allow some defects  in the color classes. A graph is {\em $(c_1, c_2, \cdots, c_k)$-colorable}  if the vertex set can be partitioned into $k$ sets $V_1,V_2, \ldots, V_k$,  such that for every $i\in [k]:=\{1, 2, \ldots, k\}$ the subgraph $G[V_i]$ has maximum  degree at most $c_i$. With this notion, a properly 3-colorable graph is $(0, 0, 0)$-colorable.  Chang, Havet, Montassier, and Raspaud~\cite{CHMR11} proved that all planar graphs without $4$-cycles or $5$-cycles are $(2,1,0)$-colorable and $(4,0,0)$-colorable. It is shown in~\cite{HSWXY13, HY13, XMW12} that  planar graphs without $4$-cycles or $5$-cycles are $(3,0,0)$- and $(1,1,0)$-colorable. Wang and his coauthors (private communication) further proved such graphs are $(2,0,0)$-colorable.

As usual, a 3-cycle is also called a {\em triangle}.  Havel (1969, \cite{H69}) asked if each planar graph with large enough distances between triangles is $(0,0,0)$-colorable. This was resolved by Dvo\"{r}\'{a}k, Kr\'{a}l and Thomas~\cite{DKT09}. We say that two cycles are {\em adjacent} if they have at least one edge in common and {\em intersecting} if they have at least one common vertex.   Borodin and Raspaud in 2003 made the following two conjectures,  which have common features with Havel's and Steinberg's 3-color problems.

\begin{conj}[Bordeaux Conjecture, \cite{BR03}] \label{con2}
Every planar graph without intersecting $3$-cycles and without $5$-cycles is $3$-colorable.
\end{conj}

\begin{conj}[Strong Bordeaux Conjecture, \cite{BR03}] \label{con3}
Every planar graph without adjacent $3$-cycles and without $5$-cycles is $3$-colorable.
\end{conj}

Let $d^{\bigtriangledown}$ denote the minimal distance between triangles.  Towards the conjectures, Borodin and Raspaud~\cite{BR03} showed that planar graphs with $d^{\bigtriangledown}\ge 4$ and without $5$-cycles are $(0,0,0)$-colorable.  This result was later improved to $d^{\bigtriangledown}\ge 3$ by Borodin and Glebov~\cite{BG04}, and independently by Xu~\cite{X07}.  Borodin and Glebov~\cite{BG11}  further improved this result to $d^\bigtriangledown\ge2$.

With the relaxed coloring notation, Xu~\cite{X08} showed that all planar graphs without adjacent 3-cycles and without $5$-cycles are $(1, 1, 1)$-colorable. Recently, Liu, Li and Yu~\cite{LLY15a, LLY15b} proved that every planar graph without intersecting 3-cycles and  without $5$-cycles is  $(2, 0, 0)$-colorable and $(1, 1, 0)$-colorable.  In this paper, we improve the results by Xu~\cite{X08} and by Liu-Li-Yu~\cite{LLY15a}, and prove the following result.

\begin{thm}\label{main1}
Every planar graph without $5$-cycles and adjacent 3-cycles is $(1, 1, 0)$-colorable.
\end{thm}

We actually prove a stronger result.  To state it, we introduce the following notion.  Let $H$ be a subgraph of $G$.  Then $(G,H)$ is {\em superextendable} if every $(1,1,0)$-coloring of $H$ can be extended to a $(1,1,0)$-coloring of $G$ such that each vertex $u\in G-H$ is coloured differently from its neighbors in $H$. If $(G,H)$ is superextendable, then we call $H$ a {\em superextendable sugraph} of $G$.  Let $\mathcal{F}$ be the family of planar graphs without $5$-cycles and adjacent $3$-cycles.

\begin{thm}\label{main}
Every triangle or $7$-cycle of a graph in $\mathcal{F}$ is superextendable.
\end{thm}

\n{\it Proof of Theorem~\ref{main1} from Theorem~\ref{main}}:\quad Let $G$ be a graph in $\mathcal{F}$. If $G$ is triangle-free,  then $G$ is 3-colorable by the Gr\"{o}ztch Theorem, and is naturally $(1, 1, 0)$-colorable; if $G$ has a  triangle, then every $(1, 1, 0)$-coloring of this triangle can be superextended to the whole graph $G$ by Theorem~\ref{main}. So Theorem~\ref{main1} follows.{\hfill \rule{2.5mm}{2.5mm}}

\medskip
Like many of the results of this kind, we also use a discharging argument to prove Theorem~\ref{main}.  The main difficulty still lies on the cases when a $4$-vertex or a $5$-vertex is incident with many triangles or many $4$-faces. Fortunately,  we could utilize many of the lemmas from  Xu~\cite{X08} and Liu-Li-Yu~\cite{LLY15a} to handle those difficult situations.

We use $G= (V, E, F)$ to denote a plane  graph with vertex set $V(G)$, edge set $E(G)$, and face set $F(G)$. For a face $f\in F(G)$, let $b(f)$ denote the boundary of a face $f$.  A $k$-vertex ($k^+$-vertex, $k^-$-vertex) is a vertex of degree $k$ (at least $k$,  at most $k$). The same notation will apply to faces and cycles.  An $(l_1, l_2, \ldots, l_k)$-face is a $k$-face $v_1v_2\ldots v_k$  with $d(v_i)=l_i$, respectively.    If a 3-vertex is incident with a triangle, then its neighbor not on the triangle is called its {\em outer neighbor}, and the $3$-face is a {\em pendant $3$-face} of its outer neighbor. Let $C$ be a cycle of a plane graph $G$. We use $int(C)$ and $ext(C)$ to denote the sets of vertices located inside and outside $C$, respectively.  A cycle $C$ is called a {\em separating cycle} if $int(C)\ne\emptyset\ne ext(C)$,  and is called a {\em nonseparating cycle} otherwise.  We also use $C$ to denote the set of vertices of $C$.

Let $S_1, S_2, \ldots, S_l$ be pairwise disjoint subsets of $V(G)$. We  use $G[S_1, S_2, \ldots, ,S_l]$ to denote the graph obtained from $G$ by contracting all the vertices in $S_i$ to a single vertex for  each $i\in\{1, 2, \ldots, l\}$.  Let $x(y)$ be the resulting vertex  by identifying $x$ and $y$ in $G$.


The paper is organized as follows. In Section $2$, we show the  reducible structures useful in our proof. In Section $3$, we are devoted to  the proof of Theorem~\ref{main} by a discharging procedure.

\section{Reducible configurations}

Suppose that $(G, C_0)$ is a counterexample to Theorem~\ref{main} with minimum  $\sigma(G)=|V(G)|+|E(G)|$, where $C_0$ is a triangle or a 7-cycle in $G$.

If $C_0$ is a separating cycle, then $C_0$ is superextendable in both  $G\setminus ext(C_0)$ and $G\setminus int(C_0)$. Hence, $C_0$ is superextendable in  $G$, contrary to the choice of $C_0$. Thus we assume that $C_0$ is the boundary of the outer face of $G$.

Let $F_k=\{f: \text{ $f$ is a $k$-face and } b(f)\cap C_0=\emptyset\}$, $F_k'=\{f:  \text{ $f$ is a $k$-face and } |b(f)\cap C_0|=1\}$, and $F_k''=\{f:  \text{ $f$ is a $k$-face and }  |b(f)\cap C_0|=2\}$.

Since $G\in \F$,  the following is immediate.

\begin{prop}\label{pr21}
 Every vertex not on $C_0$ has degree at least $3$, and no 3-face shares an edge with a 4-face in $G$.
\end{prop}

The following is a summary of some basic properties of $G$ when we consider superextendablity of a $3$-cycle or a $7$-cycle. The proofs of those results can be found, for example, in ~\cite{X08} or ~\cite{LLY15b}.

\begin{Lemma}[Xu, \cite{X08}; Liu-Li-Yu, \cite{LLY15b}]\label{le22}
The following are true about $G$:
\begin{enumerate}[(1)]
\item The graph $G$ contains neither separating triangles nor separating $7$-cycles.
\item  If $G$ has a separating $4$-cycle $C_1=v_1v_2v_3v_4v_1$, then $ext(C_1)=\{b,c\}$ such that $v_1bc$ is a $3$-cycle. Furthermore, the $4$-cycle is the unique separating $4$-cycle.
\item Let $x,y$ be two nonadjacent vertices on $C_0$. Then $xy\not\in E(G)$ and $N(x)\cap N(y)\subseteq C_0$.
\item Let $f$ be a $4$-face with $b(f)=v_1v_2v_3v_4v_1$ and let $v_1\in C_0$. Then, $v_3\not\in C_0$. Moreover, $|N(v_3)\cap C_0|=1$ if $f\in F_4''$, and $|N(v_3)\cap C_0|=0$ if $f\in F_4'$.
\item Let $u, w$ be a pair of diagonal vertices on a $4$-face. Then $G[\{u, w\}]\in \mathcal{F}$.
\end{enumerate}
\end{Lemma}

The following holds for minimum graphs that are not $(1,1,0)$-colorable.

\begin{Lemma}\label{le27}
The following are true in $G$.
\begin{enumerate}
\item (Lemma 2.5 from~\cite{HY13}) No $3$-vertex $v\not\in C_0$ is adjacent to two $3$-vertices in $int(C_0)$.
\item (Lemma 2.3 from~\cite{HY13}) $G$ has no $(3, 3, 4^-)$-face $f\in F_3$.
\item (Lemma 2.8 from~\cite{HY13}) If $v\in int(C_0)$ be a $4$-vertex incident with exactly one $3$-face that is a $(3, 4, 4)$-face in $F_3$,
then a neighbor of $v$ not on the face is either in $C_0$ or a $4^+$-vertex.
\item (Lemma 2.6 from~\cite{HY13}) Let $v\in int(C_0)$ be the $3$-vertex on a $(3, 4, 4)$-face $f\in F_3$. Then the neighbor of $v$ not on $f$ is either on $C_0$ or a $4^+$-vertex.
\item (Lemma 3(3) from~\cite{XMW12}) Suppose that $v\in int(C_0)$ is a $4$-vertex incident with two faces from $F_3$. If one of the faces is a $(3, 4, 4)$-face, then $v$ has a $5^+$-neighbor on the other face.
\end{enumerate}
\end{Lemma}

A $4$-vertex $v\in int(C_0)$ is {\em bad} if it is incident with a $(3, 4, 4)$-face from $F_3$,
A $(3, 4, 5^+)$-face from $F_3$ is {\em bad} if the $4$-vertex on it is bad.   A $5$-vertex {\em bad} if it is incident with a bad  $(3, 4, 5)$-face or a $(3, 3, 5)$-face.

\begin{Lemma}\label{le29}
Suppose that $v\in int(C_0)$ is a $5$-vertex incident with two $3$-faces $f_1$ and $f_3$ from $F_3$. Let $v_5$ be the remaining neighbor of $v$. Then each of the followings holds.
\begin{enumerate}[(1)]
\item (Lemma 5 from~\cite{XMW12}) If both $f_1$ and $f_3$ are $(3, 4^-, 5)$-faces, then $v_5$ is either on $C_0$ or a $4^+$-vertex.
\item (Lemma 4(1) from~\cite{XMW12}) At most one of $f_1$ and $f_3$ is bad.
\item (Lemma 4(2) from~\cite{XMW12}) If $f_1$ is a bad $(3, 4, 5)$-face and $f_3$ is a $(3, 4, 5)$-face,
then the outer neighbor of the $3$-vertex on $b(f_3)$ is either on $C_0$ or a $4^+$-vertex.
\item If $f_1$ is a bad $(3, 4, 5)$-face and $f_3$ is a $(4, 4, 5)$-face, then at most one $4$-vertex on $b(f_3)$ is bad.
\item \label{l212}(Lemma 8 from~\cite{XMW12}) No $6$-vertex in $int(C_0)$ is incident with three $(3, 4^-, 6)$-faces from $F_3$.
\end{enumerate}
\end{Lemma}

\begin{proof}
We only prove (4).  Let $f_1=v v_1 v_2$ with $d(v_1)= 4$ and $d(v_2)= 3$, and $f_3=v v_3 v_4$ with $d(v_3)=d(v_4)= 4$.  And let $v'_i$ and  $v''_i(\text{if any})$ be the
other neighbors of $v_i$ for $i=1, 2, 3, 4$.
Suppose that both $4$-vertices on $b(f_3)$ are bad.

Let $S=\{v, v_1,v_2, v_3, v_4, v'_1, v''_1, v'_3, v''_3, v_4', v_4''\}$, where $d(v''_1)=d(v''_3)=d(v_4'')= 4$, and let $H= G\setminus S$.  Since $\sigma(H)<\sigma(G)$, $C_0$ has a superextension
$c$ on $H$. Based on $c$,  we properly color $\{v''_1, v'_1,  v''_3, v'_3, v_3, v_4'', v_4', v, v_2\}$ in order.  Now $v_4$ can be colored as it has four properly colored neighbors.
If $v, v_4$ are colored differently, then $v_1$ can also be colored, as it has four properly colored neighbors as well.  Thus, $c(v)=c(v_4)=1$, and $v_1$ cannot be colored.
It follows that $\{c(v_1'), c(v_1'')\}=\{c(v_4'), c(v_4'')\}=\{2,3\}$.  If $c(v_3)=3$, then we can recolor $v_4$ with $2$, and color $v_1$, so let $c(v_3)=2$.  If $c(v_5)=2$,
then we recolor $v$ with $3$ and color $v_1$ with $1$ and recolor $v_2$ accordingly, so let $c(v_5)=3$.  Recolor $v$ with $2$ and color $v_1$ with $1$.  Now we can recolor $v_2$
with $1$ (if $c(v_2')=3$) or $3$ (if $c(v_2')\not=3$).
\end{proof}

For a $3$-vertex in a $3$-face $f\in F_3$, it is {\em weak} if it is adjacent to a $3$-vertex not on $f$ or $C_0$, and {\em strong} if it is adjacent to a vertex on $C_0$ or a $4^+$-vertex not on $f$.  For a vertex $v\in int(C_0)$ with $d(v)\in \{5,6\}$, $v$ is {\em weak} if  $v$ is incident with two $(5, 5^-, 3)$-faces from $F_3$ one of which is bad and adjacent to a pendant 3-face in $F_3$ when $d(v)=5$,  or  $v$ is incident to two bad $(6, 4, 3)$-faces and one $(3,5^+,6)$-face from $F_3$ when $d(v)=6$.

\begin{Lemma}\label{l210}
\begin{enumerate}[(1)]
\item There is no $(3,5^+,5^+)$-face with three weak vertices.
\item (Lemma 11 in \cite{HY13}) There is no $(3, 5^+, 5)$-face $f=uvw$ such that $u,v$ are weak and $w$ is incident with a $(5, 3, 3)$-face.
\end{enumerate}
\end{Lemma}

\begin{proof}
We only give the proof of (1) here.   Suppose that a $(3, 5^+, 5^+)$-face $f=uvw$ contains three weak vertices.  When $d(v)=5$, we label $N(v)-\{u,w\}$ as $v_1, v_2, v_3$ such that $d(v_2)=d(v_3)=3$ and $v_1$ is a bad $4$-vertex whose neighbors are $v_2, v_1', v_1''$ with $d(v_1')=4$;  when $d(v)=6$, we label $N(v)-\{u,w\}$ as $v_1, v_2, v_4, v_5$ such that $v_1, v_4$ are bad $4$-vertices with $N(v_1)=\{v_2, v_1', v_1''\}$, $N(v_4)=\{v_5, v_4', v_4''\}$ and $d(v_1')=d(v_4')=4$.  Similarly, label $N(w)-\{u, v\}$ as $w_1,w_2, w_3$.    Let $S_1=N(v)\cup \{v_1', v_1''\}$ if $d(v)=5$, and $S_1=N(v)\cup \{v_1', v_1'', v_4', v_4''\}$ if $d(v)=6$.

We first have the following claim:
\begin{equation*}\label{C1}
\text{In a $(1,1,0)$-coloring of $G-S_1$, $w$ can be properly colored.}
\end{equation*}

{\em Proof of the claim:}  Consider a $(1,1,0)$-coloring $c$ of $G-S_1$.

First let $d(w)=5$.   We may assume that $w_1, w_2, w_3$ are colored differently.  Note that we may recolor $w_3, w_1', w_1'', w_1, w_2$ in the order so that they are all properly colored.  If $c(w_3)=3$, then $\{c(w_1), c(w_2)\}=\{1,2\}$, thus we can recolor $w_2$ so that it has the same color with $w_1$. Then $w$ can be properly colored. If $c(w_3)=1$ (or $2$ by symmetry), then  $\{c(w_1), c(w_2)\}=\{2,3\}$; when $c(w_1)=2$, we can recolor $w_2$ with $2$, and color $w$ properly; when $c(w_1)=3$, $w_1', w_2''$ are colored $1$ and $2$, respectively, and we can recolor $w_1$ with $1$, then color $w$ properly.

Now assume that $d(w)=6$.  Again, we may recolor $w_1',w_1'', w_1, w_2, w_4',w_4'', w_4, w_5$ properly in the order.  If there are only two colors on $w_1, w_2, w_4, w_5$, then $w$ can be properly colored.  If $w_2$ (or $w_5$) is colored with $3$, then we can recolor it with $1$ or $2$;  if $c(w_1)=3$, then we can recolor $w_1$ with $1$ or $2$ so that it is different from the color of $w_2$.  By doing this, we may assume that $3$ is not on the four neighbors of $w$, so $w$ can be properly color with $3$.  Thus we have the claim. 

The following claim now gives a contradiction: 
\begin{quote}\label{C2}
 A $(1,1,0)$-coloring of $G-S_1$ with $w$ being properly colored can be extended to a $(1,1,0)$-coloring of $G$. 
\end{quote}

{\em Proof of the claim:}  Let $c$ be a $(1, 1, 0)$-coloring of $G-S_1$ in which $w$ is properly colored.

First assume that $d(v)=5$.   We color $u, , v_3$ properly in the order.  If $v$ can be properly colored, then we color $v_2, v_1', v_1''$ properly in the order, and finally color $v_1$, which can be color as it has only four properly colored neighbors.  If $v$ cannot be properly colored, then $w, u,v_3$ have different colors, thus $v$ can be colored $1$ or $2$. Color $v$ with $1$ for a moment.  Color $v_2, v_1', v_1''$ properly in the order. Now try to color $v_1$.  If $v_1$ is not colorable, then it must be $(c(v_1'), c(v_1''), c(v_2))=(3,2,2)$, in which case, we can color $v_1$ with $1$ and color $v$ with $2$.

Now assume that $d(v)=6$. We color $u, v, v_2, v_5, v_1', v_1'', v_4',v_4''$ properly in the order. Now $v_1, v_4$ can be colored, unless that both of them have the same color, say $1$, with $v$.  In the bad case, we can recolor $v_2, v_5$ with $1$ or $3$, then color $v$ with $2$.  Thus the claim is true and we have a contraction. 
\end{proof}

Now we discuss the configurations about $4$-faces from $F_4$.  Some of Lemmas~\ref{l213}-~\ref{l220} have their initial forms in~\cite{X08, LLY15b}.

\begin{Lemma}\label{l213}(Adapted from Lemma 3.6 in \cite{LLY15a})
\begin{enumerate}[(1)]
\item No $4$-face is from $F'_4$ in $G$.
\item Let $f\in F_4$ and let $v, x$ be a pair of diagonal vertices on $b(f)$. Then $d(v)\ge 4$ or $d(x)\ge 4$.
\end{enumerate}
\end{Lemma}

\begin{Lemma}\label{l214}
Let $v\in int(C_0)$ be a bad $4$-vertex,  or a $5$-vertex incident with a bad $(5, 4, 3)$-face, or a $5$-vertex incident with a $(5,3,3)$-face from $F_3$.  If $v$ is incident to a $4$-face $f$, then its diagonal vertex on $b(f)$ is a $4^+$-vertex.
\end{Lemma}

\begin{proof}
We consider the case  when $d(v)=5$.  The other cases are very similar and simpler.  Let $f_1=v v_1 v_2$ be a bad $(5, 4, 3)$-face with $d(v_1)= 4$, and let $f_3=v v_3 u_3 v_4$ be a $4$-face with $d(u_3)= 3$ in $G$. Let $v'_1, v''_1$ be the two other neighbors of $v_1$ with $d(v'_1)=4$ and $d(v''_1)= 3$. Let  $G'= G \setminus S$ and $H=G'[\{v_3, v_4\}]$, where $S= \{v, v_1, v_2, v'_1, v''_1\}$.  Let $v^*_3$ be the resulting vertex by identifying $v_3$ with $v_4$.  By Lemma~\ref{le22} (5), $H \in \F$.  Since $\sigma(H)< \sigma(G)$, $C_0$ has a superextension $\phi_H$ on $H$.  Based on $\phi_H$,  we color $v_3, v_4$ with the color $\phi_H(v^*_3)$ and recolor properly $u_3$ with a color in $\{1, 2, 3\}\setminus \{\phi_H(v^*_3), \phi_H(u'_3)\}$, where $u'_3$ is the other neighbor of $u_3$ in $G$. Next, properly color $v$ with a color in $\{1, 2, 3\}\setminus \{\phi_H(v^*_3), \phi_H(v_5)\}$,  and properly color $v_2$, $v'_1, v''_1$ in order, and finally color $v_1$ as it has four properly colored neighbors. Thus, $C_0$ has a superextension $\phi_G$ on $G$,  a contradiction.
\end{proof}

\medskip

For a positive integer $n$, let $[n] = \{1, 2, \ldots, n\}$.  For a vertex $v\in int(C_0)$ with $d(v)=k$, let $v_1, v_1, \ldots, v_k$ denote the neighbors of $v$ in a cyclic order. Let $f_i$ be the face with
$vv_i$ and $vv_{i+1}$ as two boundary edges for $i=1, 2, \ldots, k$, where the subscripts
are taken modulo $k$.  A $k$-vertex $v\in int(C_0)$ is {\em poor} if it is incident with $k$ $4$-faces from $F_4$, and {\em rich} otherwise. 

The following is a very useful lemma in the remaining proofs.

\begin{Lemma}\label{l215}((Lemma 3. 10 from~\cite{LLY15b})
Let $v\in int(C_0)$ be a $4$-vertex with $N(v) = \{v_i \, :\,i\in [4]\}$.  If $v$ is incident with two $4$-faces that share exactly an edge,  then no $t$-path joins $v_i$ and $v_{i+2}$ in $G$ for $t\in  \{1, 2, 3, 5\}$,  where the subscripts of $v$ are taken modulo $4$.
\end{Lemma}

\begin{Lemma}\label{l217}
Let $v\in int(C_0)$ be a $4$-vertex incident with a $4$-face $f_i=vv_iu_iv_{i+1}$. Then each of the followings holds, where the subscripts are taken modulo $4$.
\begin{enumerate}
\item If $d(v_i)=d(u_i)=3$,  then $f_{i-1}$ and $f_{i+1}$ are $6^+$-faces. Consequently, if $v$ is poor, then $v$ is not incident with $(3,3,4,4^+)$-faces. 
\item (Lemma 3.11 (1) in \cite{LLY15b}) If $f_{i+1}=vv_{i+1}u_{i+1}v_{i+2}$, then $d(u_i)\geq 4$ or $d(u_{i+1})\geq 4$.
\item (Lemma 3.11 (2) in \cite{LLY15b}) If $f_{i+2}=vv_{i+2}u_{i+2}v_{i+3}$, then $d(u_i)\geq 4$ or $d(u_{i+2})\geq 4$.
\item (Lemma 3.12 from~\cite{LLY15b}) If $v$ is a poor $4$-vertex, then either $d(v_i)\geq 5$ or $d(v_{i+2})\ge 5$.
\end{enumerate}
\end{Lemma}

\begin{proof}
(1) Suppose that $f_{i-1}$ and $f_i$ are 4-faces with $d(u_i)=d(v_i)=3$, where $v_i$'s are neighbors of 4-vertex $v$.  Identify $v_{i-1}, v_i$, and $v_{i+1}$ into one vertex, we get a new graph in $\mathcal{F}$, so the new graph is $(1,1,0)$-colorable.  Now the original graph has a $(1,1,0)$-coloring, unless $u_{i-1}$ has the same color (1 or 2), which by symmetry we assume to be $1$,  as $v_{i-1}, v_i$ and $v_{i+1}$. We uncolor $v_i$ and $v$, and then color $u_i$ and $v$ properly. Clearly, $u_i$ and $v$  are colored 2 or 3. If  $u_i$ and $v$ are colored differently, color $v_i$ with 2; if $u_i$ and $v$ are colored the same, color $v_i$ with an available color.
\end{proof}

Let $v$ be a $5^+$-vertex in $int(C_0)$. For convenience, we use $Q_4(v)$ to denote the set of poor $4$-vertices  in $N(v)\setminus C_0$ that are incident with $(3, 4, 4, 4)$-faces from $F_4$. 

\begin{Lemma}\label{l220}
Let $v$ be a poor $5$-vertex in $G$.  Then
\begin{enumerate}
\item (Lemma 3.13(2) from~\cite{LLY15b}) At most two vertices in $\{u_i\, : \, i\in [5]\}$ are $3$-vertices.
\item (Lemma 3.13(3) from~\cite{LLY15b}) If $d(u_i)= 3$, then either $d(v_{i-1})\geq 5$ or $d(v_{i+2})\geq 5$.
\item (Lemma 3.13(1) from~\cite{LLY15b}) If $d(u_i)=d(v_i)= 3$, then $d(u_{j})\geq 4$ for $j\in [5]\setminus \{i\}$.
\item If $f_i$ is a $(5, 4, 3, 4)$-face, then at most one of $v_i, v_{i+1}$ is in $Q_4(v)$.
\item If $d(v_i)=d(u_i)=d(v_{i+2})=3$, then $d(v_{j})\geq 5$ for $j\in [5]\setminus \{i, i+2\}$.
\item If $f_i$ is a $(5, 3, 4, 4)$-face such that $u_i$ and $v_{i+1}$ are poor $4$-vertices, then $v_{i+1}\not\in Q_4(v)$.
\end{enumerate}
\end{Lemma}

\begin{proof}

(4) Without loss of generality, assume that $f_1=vv_1u_1v_2$ is a $(5, 4, 3, 4)$-face. Assume further that $v_1, v_2$ are in $Q_4(v)$. By Lemma~\ref{l217}(4),$d(u_2)\geq 5$ and $d(u_5)\geq 5$ as $d(u_1)=3$. Let $N(u_1)=\{v_1, v_2, w\}$ and $N(w)\cap N(v_1)=\{v_1'\}$ and $N(w)\cap N(v_2)=\{v_2'\}$.  It implies that  $u_1v_1v_1
w$ and $u_1v_2v_2'w$ are $(3,4,4,4)$-faces.  So $d(v_1')=d(v_2')=d(w)=d(v_1)=d(v_2)=4$.  Let $H= G[\{u_1, v_1', v_2', v\}]$. By Lemmas~\ref{le22} (5) and~\ref{l215},  $H\in \mathcal{F}$.  Let $c$ be a coloring of $(H,C_0)$. Let $v'$ be the resulting vertex of the identification. In $G$, color $u_1, u_2, u_3$ with $c(v')$, and properly color $v_1, v_2, w$ in the order, and finally color $u_1$.  Thus, $(G, C_0)$ is superextendable, a contradiction.

(5) By symmetry, assume that $d(v_1)=d(u_1)=d(v_3)=3$. By Lemma~\ref{l214} (2), $d(v_2)\ge 4$, and furthermore, by Lemma~\ref{l217} (2), $d(v_2)\geq 5$.

We show that $d(v_5)\geq 5$.  Suppose otherwise that $d(v_5)\le 4$. Then by Lemma~\ref{l213}(2), $d(v_5)=4$.  Let $G'=G-\{v,v_5\}$ and $H=G'[\{u_4, u_5\}, \{v_2, v_4\}]$.  By Lemmas~\ref{le22} (5) and~\ref{l215},  $H\in \mathcal{F}$. Then $(H,C_0)$ is superextendable, and let $c$ be a coloring of $(H,C_0)$.  In $G$, color $v_2, v_4$ and $u_4, u_5$ with the colors of the identified vertices, respectively, then properly recolor $v_5, u_1, v_1, v_3$ in the order. Let $c'$ be the resulting coloring of $G-v$.  Now we color $v$.  If $c'(v_2)=c'(v_4)=3$, then $v$ can be colored, as the other three colored neighbors are all properly colored, so we may assume that $c'(v_2)=c'(v_4)=1$. If $c'(v_1)=1$, then clearly $v$ can be colored; If $c'(v_1)=3$, then we  uncolor $v_1$ and color $v$ with $3$ (if $c'(v_3), c'(v_5)\not=3$) or with $2$ (if $c'(v_3)=3$ or $c'(v_5)=3$), and now $v_1$ can be colored, as $c'(u_1)\in \{1,2\}$ and $u_1$ is properly colored, a contradiction.


Similarly, we have $d(v_4)\ge 5$.  


(6)  As $v_{i+1}$ is a poor $4$-vertex and $d(u_i)=4$, by Lemma~\ref{l217}(4), $d(u')\ge 5$, where $u'$ is the diagonal vertex to $v$ on the $4$-face incident with $v_{i+1}$; similarly, $u_i$ is a poor $4$-vertex and $d(v_i)=3$, by Lemma~\ref{l217}(4), $d(u_1')\ge 5$, where $u_1'$ is the diagonal vertex to $v_2$ on the $4$-face incident with $u_1$. It follows that no $4$-face incident with $v_{i+1}$ is a $(3,4,4,4)$-face, so $v_{i+1}\not\in Q_4(v)$.
\end{proof}

\section{Discharging procedure}
In this section, we prove the main theorem by a discharging argument.

Let the initial charge of a vertex $v$ be $\mu(v)=2d(v)-6$,  the initial charge of a face $f\not=C_0$ be $\mu (f)=d(f)-6$, and $\mu(C_0)= d(C_0)+6$. By Euler's formula,  $\sum_{x\in V\cup F}\mu (x)=0$.   

A $(3, 4, 4, 5^+)$- or $(3, 4, 5^+, 4)$-face $f\in F_4$ is {\em superlight} if both $4$-vertices on $b(f)$ are poor and {\em light} otherwise.


\medskip

The following are the discharging rules:

\begin{enumerate}[(R1)]

\item[(R1)]  Let $v\in int(C_0)$ with $d(v)=4$ and $f\in F_3\cup F_4$ be a face incident with $v$.
\begin{enumerate}
\item[(R1.1)] When $f\in F_3$,  $f$ gets $1$ from $v$, unless $v$ is incident with $f$ and a $(3,4,4)$-face $f'\in F_3$, in which case, $v$ gives $\frac{5}{4}$ to $f'$ and $\frac{3}{4}$ to $f$.  

\item[(R1.2)] When $f\in F_4$, $f$ gets $1$ from $v$ if it is a $(4, 3, 3, 4^+)$-face,
$\frac{2}{3}$ if it is a $(4, 4, 4, 3)$-face or $v$ is rich,  and $\frac{1}{2}$ otherwise.
\end{enumerate}

\item[(R2)] Let $v\in int(C_0)$ with $d(v)\ge 5$.
\begin{enumerate}

\item[(R2.1)] Let $f=vuw$ be a $3$-face in $F_3$ incident with $v$.
\begin{enumerate}[(R2a)]
\item Let $f$ be a $(5^+, 3,3)$-face.  Then $v$ gives $2$ to $f$.

\item Let $f$ be a $(5^+, 4, 3)$-face. Then $v$ gives $\frac{9}{4}$ to $f$  if $u$ is bad and $w$ is weak; $2$ to $f$ if $u$ is not bad and $w$ is weak;  $\frac{7}{4}$ to $f$ if $u$ is bad and $w$ is strong; $\frac{3}{2}$ to $f$ if $u$ is not bad and $w$ is strong. 

\item Let $f$ be a $(5^+, 5^+, 3)$-face. Then $v$ gives $\frac{5}{4}$ to $f$ if $v$ is weak, and $\frac{7}{4}$ otherwise.

\item Let $f$ be a $(5^+, 4^+, 4^+)$-face. Then $v$ gives $\frac{3}{2}$ to $f$ if both $u$, $w$ are bad $4$-vertices; $\frac{5}{4}$ to $f$ if exactly one of $u$ and $w$ is a bad $4$-vertex; $1$ to other $(5^+, 4^+, 4^+)$-faces.
\end{enumerate}

\item[(R2.2)] For each $4$-face $f\in F_4$ incident with $v$, $v$ gives $1$ to $f$ if $f$ is a $(5^+, 4^+, 3, 3)$-face or a superlight $(3,4,4,5^+)$- or $(3,4,5^+,4)$-face;  $\frac{5}{6}$ to $f$ if $f$ is a light $(3,4,4,5^+)$- or $(3,4,5^+,4)$-face;  $\frac{3}{4}$ to $f$ if $f$ is a $(3, 4, 5^+, 5^+)$-face or a $(3, 5^+, 4, 5^+)$-face; $\frac{1}{2}$ to $f$ if $f$ is a $(4^+, 4^+, 4^+, 5^+)$-face.

\item[(R2.3)] If $Q_4(v)\neq \emptyset$, then $v$ gives $\frac{1}{6}$ to each $4$-vertex in $Q_4(v)$.
\end{enumerate}

\item[(R3)] Each $4^+$-vertex sends $\frac{1}{2}$ to each of its pendant 3-faces in $F_3$.

\item[(R4)] Let $v\in C_0$. Then $v$ gives $3$ to
each incident $3$-face from $F'_3$; $\frac{3}{2}$ to each incident face from $F''_3$;
$1$ to each incident $4$-face from $F''_4$.

\item[(R5)] $C_0$ gives $2$ to each $2$-vertex on $C_0$; $\frac{3}{2}$ to each
$3$-vertex on $C_0$;  $1$ to each $4$-vertex on $C_0$; and $\frac{1}{2}$ to each $5$-vertex on $C_0$.  In addition,  if $C_0$ is a $7$-face with six $2$-vertices, then it gets $1$ from the incident face.
\end{enumerate}

We will show that each $x\in F(G)\cup V(G)$ has final  charge $\mu^*(x)\ge 0$ and at least one face has positive charge, to reach a contradiction.

As $G$ contains no $5$-faces, and $6^+$-faces  other than $C_0$ are not involved in the discharging procedure, we will check the final charge of the $3$- and $4$-faces other than $C_0$ first.

\begin{Lemma}\label{le31}
Let $f$ be a $i$-face in $F(G)\setminus C_0$ for $i=3, 4$. Then $\mu^*(f)\geq 0$.
\end{Lemma}

\begin{proof}
Suppose that $d(f)=3$ and $f=v u w$ with $d(v)\leq d(u)\leq d(w)$.   By Lemma~\ref{le22} (3), $|b(f)\cap C_0|\leq 2$. If $|b(f)\cap C_0|= 2$,  then $f\in F''_3$, by (R4), $\mu^*(f)\geq -3+2\times \frac{3}{2}=0$; if $|b(f)\cap C_0|= 1$,  then $f\in F'_3$, by (R4), $\mu^*(f)\geq -3+3=0$. Hence, let $|b(f)\cap C_0|= 0$.  By Proposition~\ref{pr21}, $d(v)\geq 3$.

Assume first that $d(v)=3$. If $f$ is a $(3, 3, a)$-face, by Lemma~\ref{le27} (2), $a\geq 5$ and the outer neighbors  of $u, v$ are of degree at least $4$ or  on $C_0$, then by (R2a) and (R3),  $\mu^*(f)\geq -3+2\times \frac{1}{2}+2=0$.  If $f$ is a $(3, 4, 4)$-face, by Lemma~\ref{le27} (4), the third neighbor of $v$ is a $4^+$-vertex or  on $C_0$, then by (R1.1) and (R3), $\mu^*(f)\geq -3+2\times \frac{5}{4}+\frac{1}{2}=0$.  Now let $f$ be a $(3, 4, 5^+)$-face. Then by (R1.1) and (R2b), $\mu^*(f)\ge -3+\frac{9}{4}+\frac{3}{4}=0$  if $v$ is weak and $u$ is bad; $\mu^*(f)\ge -3+\frac{7}{4}+\frac{3}{4}+\frac{1}{2}=0$ if $v$ is strong and $u$ is bad;  $\mu^*(f)\ge -3+2+1=0$ if $v$ is weak and $u$ is not bad;  $\mu^*(f)\ge -3+\frac{3}{2}+1+\frac{1}{2}=0$ if $v$ is strong and $u$ is not bad. 

Assume that $d(v)=4$. Then $d(w)\geq d(u)\geq 4$.  If $f$ is a $(4, 4, 4)$-face, then by  Lemma~\ref{le27} (5),  none of the $4$-vertices on $f$ can be bad,  thus by (R1.1), $\mu^*(f)\geq -3+3\times 1=0$. Now assume that $f$ is a $(4, 4^+, 5^+)$-face.  In this case,  if $v, u$ are two bad $4$-vertices, 
then by (R1.1) and (R2c), $f$ receives at least $\frac{3}{2}$ from $w$ and at least $\frac{3}{4}$ from each of $v$ and $u$, thus  $\mu^*(f)\geq -3+2\times \frac{3}{4}+\frac{3}{2}=0$; if one of $v, u$ is not bad, then by (R1.1) and (R2c),  $w$ gives at least $\frac{5}{4}$ to $f$ and $u, v$ give at least $(\frac{3}{4}+1)$ to $f$, then  $\mu^*(f)\geq -3+(\frac{3}{4}+1)+\frac{5}{4}=0$; for other cases, by (R1.1), and (R2c), $f$ receives at least $1$ from each vertex on $b(f)$, thus  $\mu^*(f)\geq -3+3\times 1=0$.

Finally, let $d(v)\ge 5$.  It follows that $f$ is a $(5^+, 5^+, 3)$-face, then by Lemma~\ref{l210} (1),  at most two of the three vertices are weak, so by (R2c) and (R3),  $\mu^*(f)\geq -3+\min\{\frac{5}{4}+ \frac{7}{4}, 2\times \frac{7}{4}, 2\times\frac{5}{4}+\frac{1}{2}\}\ge 0$.

\medskip
Suppose that $d(f)= 4$ and $f=vuwx$.  By Lemma~\ref{le22} (4), $|b(f)\cap C_0|\leq 2$. If $|b(f)\cap C_0|= 2$,  then $f\in F''_4$, by (R4), $\mu^*(f)\geq -2+2\times 1=0$. By Lemma~\ref{l213}(1)  $F'_4=\emptyset$. Hence, assume that $|b(f)\cap C_0|= 0$.  By Proposition~\ref{pr21}, $d(z)\geq 3$ for each $z\in b(f)$.  By Lemma~\ref{l213} (2), if $d(z)=3$ for some $z\in b(f)$, then its diagonal vertex on  $b(f)$ is a $4^+$-vertex.

If $f$ is a $(3, 3, 4^+, 4^+)$-face, then by (R1.2) and (R2.2), $\mu^*(f)\geq -2+2\times 1=0$. If $f$ is a $(3, 4, 4, 4)$-face, then by (R1.2), $\mu^*(f)\geq -2+3\times \frac{2}{3}=0$. If $f$ is a $(3, 4^+, 5^+, 5^+)$-face or $(3, 5^+, 4^+, 5^+)$-face, then by (R2.2), $\mu^*(f)\ge -2+2\cdot \frac{3}{4}+\frac{1}{2}=0$.  
If $f$ is a $(4^+, 4^+, 4^+, 4^+)$-face, then by (R1.2) and (R2.2), $\mu^*(f)\geq -2+4\times \frac{1}{2}=0$.
Finally, let $f$ be a $(3, 4, 4, 5^+)$-face or $(3, 4, 5^+, 4)$-face. If $f$ is superlight, then $f$ gets $1$ from the $5^+$-vertex on $b(f)$ and at  least $\frac{1}{2}$ from each $4$-vertex on $b(f)$,  thus $\mu^*(f)\geq -2+1+2\times \frac{1}{2}=0$. Otherwise, $f$ is light, by (R2.2),  $f$ receives $\frac{5}{6}$ from the $5^+$-vertex,  $\frac{2}{3}$ from a rich $4$-vertex and  at least $\frac{1}{2}$ from the other $4$-vertex on $b(f)$,  then $\mu^*(f)\geq -2+\frac{5}{6}+\frac{2}{3}+\frac{1}{2}=0$.
\end{proof}

\medskip

Let $v$ be a $k$-vertex in $int(C_0)$. Let $t_i$ be the  number of $i$-faces incident with $v$ in $F_i$ for $i\in \{3,4\}$. Let $t_p$  be the number of pendant $3$-faces adjacent to $v$. By Proposition~\ref{pr21}, 
\begin{equation}\label{eq1}
t_3\leq \lfloor \frac{k}{2}\rfloor,\; \text{and}\;\ t_4\leq \text{max}\{0, k-2t_3-t_p-1\} \text{ if $t_4>0$.}
\end{equation}


\begin{Lemma}\label{le32}
Let $v\in int(C_0)$ be a $4$-vertex. Then $\mu^*(v)\geq 0$.
\end{Lemma}

\begin{proof}
If $N(v)\cap C_0\not=\emptyset$, then $t_3\le 1$, thus $\mu^*(v)\ge 2-\max\{\frac{5}{4}+\frac{1}{2}, 2\cdot 1, 3\cdot \frac{1}{2}\}\ge 0$. So, let $N(v)\cap C_0=\emptyset$. Clearly, $t_3\le 2$.  

If $t_3= 2$, then by Lemma~\ref{le27} (5), at most one of the triangles is a $(3,4,4)$-face, thus by (R1.1), $\mu^*(v)\geq 2-\max\{\frac{5}{4}+\frac{3}{4}, 2\cdot 1\}=0$. If $(t_3, t_4)=(1,1)$, then when $v$ is not bad, by (R1.1) and (R1.2), $v$ gives at most one to each of the incident faces, thus $\mu^*(v)\ge 2-2\cdot 1=0$, and when $v$ is bad,  $v$ cannot be incident with a $(3, 3, 4, 4^+)$-face by Lemma~\ref{l214}, then by (R1.1) and (R1.2), $\mu^*(v)\geq  2-\frac{5}{4}-\frac{2}{3}=\frac{1}{12}>0$.  Let $(t_3, t_4)=(1,0)$.  Then $0\le t_p\leq 2$. By Lemma~\ref{le27} (3), at least one of the other neighbors of $v$ is a $4^+$-vertex or in $C_0$ when $v$ is bad, thus by (R1.1) and (R3), $\mu^*(v)\geq 2-\max\{\frac{5}{4}+\frac{1}{2}, 1+2\times\frac{1}{2}\}=0$.

Now, we assume that $t_3= 0$.    If $t_p\ge 2$, then $t_4\leq 1$, so by (R1.2) and (R3), $\mu^*(v)\geq 2- \max\{4\cdot \frac{1}{2}, 1+2\times \frac{1}{2}\}=0$. Assume that $t_p=1$ and $t_4=2$. Let $v$ be incident with $4$-faces $f_3=v v_2 u_2 v_3$ and  $f_4=v v_3 u_3 v_4$ in $F_4$.  By Lemmas~\ref{l213} (2) and~\ref{l217},  at most two of the vertices in $\{v_2, u_2, v_3, u_3, v_4\}$ are $3$-vertices, and when $d(v_3)=3$, none of the other vertices is a $3$-vertex,  then by (R1.2) and (R3), $v$ gives at most $\max\{2\cdot\frac{2}{3}, 1+\frac{1}{2}\}=\frac{3}{2}$ to $f_3$ and $f_4$, thus $\mu^*(v)\geq 2-\frac{3}{2}-\frac{1}{2}=0$.

Lastly, let $t_3=t_p=0$. If $t_4\le 2$, by (R1.2), $\mu^*(v)\geq 2-2\times 1=0$.  Let $t_4= 3$.  If $v$ is not incident with a $(4, 3, 3, 4^+)$-face, then by (R1.2), $\mu^*(v)\geq 2-3\times \frac{2}{3}=0$;  If  $v$ is incident with a $(4, 3, 3, 4^+)$-face, then by Lemma~\ref{l217}, the other incident $4$-faces are $(4, 4^+, 4^+, 4^+)$-faces, so by (R1.2), $\mu^*(v)\geq 2-1-2\times \frac{1}{2}=0$.  Hence assume that $t_4= 4$, that is, $v$ is poor.   By Lemma~\ref{l217} (4), $v$ is adjacent to at least two $5^+$-vertices, and without loss of generality, let $d(v_3), d(v_4)\ge 5$. By Lemma~\ref{l217} (1), $v$ is not incident with $(3,3,4,4^+)$-faces.  Thus, if $v$ is not incident with $(3,4,4,4)$-faces, then by (R1), $\mu^*(v)\ge 2-4\times \frac{1}{2}=0$, and if $v$ is incident with a $(3,4,4,4)$-face, then by (R2.3), $v$ gets $\frac{1}{6}$ from each of its $5^+$-neighbors, so by (R1) and (R2.3), $\mu^*(v)\geq 2-3\times \frac{1}{2}-\frac{2}{3}+2\cdot \frac{1}{6}>0$.
\end{proof}


\begin{Lemma}\label{le33}
Let $v\in int(C_0)$ be a $k$-vertex with $k\geq 5$. If $u\in Q_4(v)$, then one of the $4$-faces that contain $uv$ as an edge contains no $3$-vertices or is a $(3, 5^+,4,5^+)$-face.
\end{Lemma}

\begin{proof}
As $u\in Q_4(v)$,  $u$ is a poor $4$-vertex and incident with one $(3, 4, 4, 4)$-face. Suppose that  $f_i=u v_i u_i v_{i+1}$ for $i\in [4]$, where $v= v_4$ and the subscripts are taken modulo $4$.  We show that $f_3$ or $f_4$ contains no $3$-vertices or is a $(3, 5^+,4,5^+)$.

If $d(u_3)\geq 4$ and $d(u_4)\geq 4$, then  by Lemma~\ref{l217}(4), either $d(v_1)\ge 5$ or $d(v_3)\geq 5$,  so $f_3$ or $f_4$ cannot contain $3$-vertices.  Thus, by symmetry, let $d(u_3)= 3$.  By Lemma~\ref{l217} (1)-(3), $d(v_3)\geq 4$ and $d(u_j)\ge 4$ for $j\in [4]\setminus \{3\}$. So $f_4$ contains no $3$-vertices if $d(v_1)\not=3$.  Let $d(v_1)=3$. Then $v_1uv_2u_1$ is the $(3,4,4,4)$-face. By Lemma~\ref{l217}(4), $d(v_3)\geq 5$, so $f_3$ is a $(3,5^+,4,5^+)$-face, as desired. 
\end{proof}


Let $v\in int(C_0)$ with $d(v)=k\geq 5$. 
By Lemma~\ref{le33}, a vertex in $Q_4(v)$ must either share a $4$-face without $3$-vertices with $v$, or is on a $(3,5^+,4,5^+)$-face.  In the former case, the $4$-face could contain at most two vertices from $Q_4(v)$, then the charges from $v$ to the vertices and the $4$-face are at most $\frac{1}{2}+2\cdot \frac{1}{6}<1$. In the latter case, the face contains exactly one vertex from $Q_4(v)$, then by (R2), the charges from $v$ to the vertex and the $4$-face are at most $\frac{3}{4}+\frac{1}{6}<1$.  Thus,  by (R2),
\begin{align}
\mu^*(v) &\geq 2k-6-\frac{9}{4}t_3-t_4-\frac{1}{2}t_p\label{eq2}\\
        &\geq (k-2t_3-t_4-t_p)+(\frac{7}{8}k-6)\ge \frac{7}{8}k-6\; \text{ (as $t_3\leq \lfloor\frac{k}{2}\rfloor$)}.\label{eq3}
\end{align}

\begin{Lemma}\label{le34}
Suppose that $v\in int(C_0)$ is a $5$-vertex with $t_3> 0$. Then $\mu^*(v)\geq 0$.
\end{Lemma}

\begin{proof}
 If $|N(v)\cap C_0|\geq 2$, then $t_3+t_p\leq 2$ and $t_4=0$, as $t_3>0$, so by (R1)-(R5), $\mu^*(v)\ge 4-9/4-1/2>0$.  If $|N(v)\cap C_0|=1$, then $v$ cannot be incident with two bad 3-faces in $F_3$ by Lemma~\ref{le29} (2), and when $v$ is incident with a bad $3$-face and a $(3,4,5)$-face, the $3$-vertex is strong on the $(3,4,5)$-face by Lemma~\ref{le29} (3),  so $\mu^*(v)\geq 4-\max\{\frac{9}{4}+\frac{1}{2}, \frac{9}{4}+\frac{7}{4}, 2\times 2\}=0$ by (R2.1)-(R2.3) and (R3).   Therefore, we assume that $|N(v)\cap C_0|=0$.


Assume first that $t_3=2$.  Let $f_1=vv_1v_2$ and $f_3=vv_3v_4$ be the incident $3$-faces and $v_5$ be the fifth neighbor of $v$.  By Lemma~\ref{le29} (2), at most one of $f_1, f_3$  is bad.   If both $f_1$ and $f_3$ are $(3, 4^-, 5)$-faces, then by Lemma~\ref{le29} (1),  $d(v_5)\ge 4$ or $v_5\in C_0$, , and by Lemma~\ref{le29} (3), if one is bad, then the $3$-vertex on the other one is strong, thus by (R2b),  $\mu^*(v)\geq 4-\max\{\frac{9}{4}+\frac{7}{4}, 2\times 2\}= 0$.  If $f_1$ is a $(3,4,5)$-face and $f_3$ is a $(3, 5, 5^+)$-face, then by (R2) and (R3), $\mu^*(v)\geq 4-(\frac{9}{4}+\frac{5}{4}+\frac{1}{2})= 0$ if $v$ is weak, and $\mu^*(v)\ge 4-\max\{\frac{9}{4}+\frac{7}{4}, \frac{7}{4}+\frac{7}{4}+\frac{1}{2}\}\ge 0$ if $v$ is not weak. If none of $f_1,f_3$ is a $(3,4,5)$-face, then by (R2), $\mu^*(v)\ge 4-2\cdot \frac{7}{4}-\frac{1}{2}=0$.

Finally, let $t_3= 1$. Then $t_4\leq 2$. If $t_4\leq 1$, then $|Q_4(v)|=0$,  thus, by (R2a), (R2.2) and (R3),  $\mu^*(v)\geq 4-\frac{9}{4}-1-\frac{1}{2}=\frac{1}{4}>0$.  Thus assume that $t_4= 2$ and let $f_1=vv_1v_2$, $f_3=vv_3u_3v_4$ and $f_4=vv_4u_4v_5$ be the incident faces.  Note that $v_3, v_5$ are rich and $|Q_4(v)|\leq 1$. If  $f_1$ is not bad, then by Lemma~\ref{le33} and by (R2.1), (R2.2)  and~(R2.3),  $\mu^*(v)\geq 4-2-\max\{1+\frac{3}{4}+\frac{1}{6}, 2\cdot 1\}=0$. Therefore,  let $f_1$ be a bad $(5, 4, 3)$-face.  By Lemma~\ref{l214}, $d(u_3)\geq 4$ and $d(u_4)\geq 4$.  Consider $d(v_4)=3$ first.  Then $|Q_4(v)|= 0$, and by Lemma~\ref{l213}, $d(v_3)\geq 4$ and $d(v_5)\geq 4$.  Therefore,  if $d(v_3)=d(v_5)=4$, then as $v_3, v_5$ are rich, by (R2a) and (R2.2),  $v$ gives at most $\frac{5}{6}$ to each of $f_3, f_4$, thus, $\mu^*(v)\geq 4-\frac{9}{4}  -2\times \frac{5}{6}=\frac{1}{12}>0$; if $d(v_3)\geq 5$ or $d(v_5)\geq 5$,  then there are at least two $5^+$-vertices in $b(f_3)$ or $b(f_4)$,  thus, by (R2a) and (R2.2), $\mu^*(v)\geq 4-\frac{9}{4}-1-\frac{3}{4}=0$.   Assume next that $d(v_4)= 4$. Then  $|Q_4(v)|\leq 1$. By Lemma~\ref{l217} (2), either $d(v_3)\geq 4$ or $d(v_5)\geq 4$.  It means that $f_3$ or $f_4$ is a $(5, 4, 4^+, 4^+)$-face, by (R2a), (R2.2) and (R2.3),  $\mu^*(v)\geq 4-\frac{9}{4}-1-\frac{1}{2}-\frac{1}{6}=\frac{1}{12}>0$.   Assume that $d(v_4)\geq 5$. Then $|Q_4(v)|=0$, and $f_3, f_4$ are $(5, 5^+, 4^+, 3^+)$-faces,  by (R2a) and (R2.2), $v$ gives at most $\frac{3}{4}$ to each of $f_3, f_4$, then  $\mu^*(v)\geq 4-\frac{9}{4}-2\times \frac{3}{4}=\frac{1}{4}>0$.
\end{proof}

\medskip

For a poor $5$-vertex $v\in int(C_0)$,  let $f(v)=(f_1, f_2, f_3, f_4, f_5)$, where $f_i=vv_iu_iv_{i+1}$ with $i\in \Z_5$, the cyclic group of order $5$.  We say that $v$ gives a {\em charge sequence $(a_1, a_2, a_3, a_4, a_5)$ to $f(v)$} if $v$ gives  at most $a_i$ to $f_i$ by (R2.2).

\begin{Lemma}\label{le35}
For each $5$-vertex $v\in int(C_0)$,  $\mu^*(v)\geq 0$.
\end{Lemma}

\begin{proof}
By Lemma~\ref{le34}, we may assume  that $t_3=0$.   By~(\ref{eq1}), $\mu^*(v)\ge 4-t_4-t_p/2\ge 0$ if $t_4\le 4$.  Thus, we let $t_4= 5$, that is, $v$ is poor. Let $M(v)=\{u_1, u_2, u_3, u_4, u_5\}$.  By Lemma~\ref{l220}(1), $M(v)$ has at most two $3$-vertices, and by Lemma~\ref{l213} (2),  there are at most two $3$-vertices in $N(v)$.

\medskip
\n{\bf Case 1.} $N(v)$ has exactly two $3$-vertices.
\medskip

By symmetry and Lemma~\ref{l213} (2), we may assume that $d(v_1)= d(v_3)= 3$. 
By Lemma~\ref{l213} (2), $d(v_2), d(v_4), d(v_5)\ge 4$.  Furthermore, by Lemma~\ref{l217} (2), $d(v_2)\not=4$, thus $d(v_2)\ge 5$.

Assume that some vertex, say $u_1$,  in $M(v)$ has degree $3$.  Then by Lemma~\ref{l220} (3), $d(u_j)\geq 4$ for $j\in [5]\setminus \{1\}$, and by Lemma~\ref{l220}(5),  $d(v_j)\ge 5$  for $j\in [5]\setminus \{1, 3\}$, thus $|Q_4(v)|=0$, and $f_2, f_3, f_4, f_5$ are $(5, 5^+, 4^+, 3)$-   $(5, 3, 4^+, 5^+)$-, $(5, 4^+, 4^+, 5^+)$-, and $(5, 5^+, 4^+, 3)$-faces, respectively.   By (R2.2) and (R2.3), $v$ gives a charge sequence $(1, \frac{3}{4}, \frac{3}{4}, \frac{1}{2}, \frac{3}{4})$  to $f(v)$, thus $\mu^*(v)\geq 4-1-3\times\frac{3}{4}-\frac{1}{2}=\frac{1}{4}>0$.  Hence we assume that $M(v)$ has no $3$-vertex.

As $f_1, f_2$ are $(3,4^+,5^+,5)$-faces and $f_4$ is a $(4^+, 4^+, 4^+,5)$-face, by (R2.2), $v$ gives $2\cdot\frac{3}{4}+\frac{1}{2}=2$ to them.  Consider $f_3$ (and similarly $f_5$), which is a $(3,4^+,4^+,5)$-face. We claim that  $$\text{$v$ gives at most $1$ to the face and the vertex in $Q_4(v)\cap b(f_3)$, }$$ which shows that $\mu^*(v)\ge 4-2-2\cdot1=0$. 
In fact, if it contains another $5^+$-vertex, then by (R2.2) and (R2.3), $v$ gives at most $\frac{3}{4}+\frac{1}{6}<1$, as desired.  So let it be a $(3,4,4,5)$-face.  If it contains two poor $4$-vertices, then it is superlight and by Lemma~\ref{l220}(4), it contains no vertex in $Q_4(v)$, thus by (R2.2), $v$ gives $1$ to it; otherwise, it is light, thus by (R2.2) and (R2.3), $v$ gives $\frac{5}{6}+\frac{1}{6}=1$ to it. 


\medskip
\n{\bf Case 2.} $N(v)$ has exactly one $3$-vertex.  By symmetry, we assume that $d(v_1)= 3$.
\medskip

Assume first that $M(v)$ contains no $3$-vertex.  Then each of $f_2, f_3,f_4$ is a $(5, 4^+, 4^+, 4^+)$-face, and $f_1$ is a $(5, 3, 4^+, 4^+)$-face  and $f_5$ is a $(5, 4^+, 4^+, 3)$-face. By (R2.2) and (R2.3),  $v$ gives a charge sequence $(1, \frac{1}{2}, \frac{1}{2},\frac{1}{2}, 1)$ to $f(v)$, thus $\mu^*(v)\geq 4-2\times 1-3\times \frac{1}{2}-3\times \frac{1}{6}=0$ if $|Q_4(v)|\le 3$.  Thus, assume that $v_j\in Q_4(v)$ for $j\in [5]\setminus \{1\}$.  If $u_1$ is a poor $4$-vertex, then   $f_1=vv_1u_1v_2$ is a $(5,3,4,4)$-face  such that $u_1,v_2$ are poor and $v_2\in Q_4(v)$, a contradiction to Lemma~\ref{l220}(6). Thus, assume that $u_1$ is not a poor 4-vertex. In this case, $f_1$ is light. By (R2.2), $v$ gives $\frac{5}{6}$ to $f_1$. Thus, $\mu^*(v)\geq 4-1-\frac{5}{6}-3\times \frac{1}{2}-4\times \frac{1}{6}=0$.

\medskip

Next, assume that $M(v)$ contains exactly one $3$-vertex. Let $d(u_1)= 3$ (or by symmetry $d(u_5)=3$).  By our assumption, $d(u_j)\geq 4$ for $j\not= 1$ and $j\in [5]$. Thus, each of $f_2, f_3,f_4$ is a $(5, 4^+, 4^+, 4^+)$-face, $f_5$ is a $(5, 4^+, 4^+, 3)$-face. Note that if $d(v_2)\geq 5$, then $|Q_4(v)|\leq 3$; if $d(v_2)= 4$,  then by Lemma~\ref{l217}(1), $v_2$ is rich,  which implies that $v_2\not\in Q_4(v)$, then $|Q_4(v)|\leq 3$.  By (R2.2) and (R2.3), $v$ gives a charge sequence $(1, \frac{1}{2}, \frac{1}{2},\frac{1}{2}, 1)$ to  $f(v)$, and $\mu^*(v)\geq 4-2\times 1-3\times \frac{1}{2}-3\times \frac{1}{6}=0$.  Hence, we may assume, by symmetry, that either $d(u_3)=3$ or $d(u_2)=3$.

Let $d(u_3)=3$.  By Lemma~\ref{l220} (2), either $d(v_2)\geq 5$ or $d(v_5)\geq 5$, and by symmetry we may assume that $d(v_5)\ge 5$.  This implies that $f_5$ is a $(3, 5, 4^+, 5^+)$-face and $v_5\notin Q_4(v)$.  In this case,  both $f_2, f_4$ are two $(5, 4^+, 4^+, 4^+)$-faces, and $f_1$ is a $(5, 3, 4^+, 4^+)$-face.  If $d(v_3)\geq 5$ or $d(v_4)\geq 5$, then $|Q_4(v)|\leq 2$, and by (R2.2) and (R2.3), $v$ gives a charge sequence  $(1, \frac{1}{2}, \frac{3}{4}, \frac{1}{2}, \frac{3}{4})$ to $f(v)$, thus,  $\mu^*(v)\geq 4-1-2\times \frac{3}{4}-2\times \frac{1}{2}-2\times \frac{1}{6}=\frac{1}{6}>0$.   Then assume that $d(v_3)=d(v_4)= 4$. By Lemma~\ref{l220}(4), either $v_3\not\in Q_4(v)$   or $v_4\not\in Q_4(v)$. If $f_1$ is a $(5,3,4,4)$-face with two poor $4$-vertices, then by Lemma~\ref{l220}(6), $v_2\not\in Q_4(v)$, it follows that $|Q_4(v)|\le 1$, so by (R2.2) and (R2.3),  $v$ gives a charge sequence $(1, \frac{1}{2},  1,\frac{1}{2},\frac{3}{4})$ to $f(v)$, thus, $\mu^*(v)\geq 4-1-\frac{1}{2}-1-\frac{1}{2}-\frac{3}{4}-\frac{1}{6}=\frac{1}{12}>0$;  otherwise, $f_1$ is a light $(5,3,4^+,4^+)$-face, so by (R2.2) and (R2.3), $\mu^*(v)\ge 4-\frac{5}{6}-\frac{1}{2}-1-\frac{1}{2}-\frac{3}{4}-2\cdot \frac{1}{6}>0$.

Let $d(u_2)=3$ now. By Lemma~\ref{l220} (2),  $d(v_4)\geq 5$. Then both $f_3, f_4$ are $(5, 4^+, 4^+, 4^+)$-faces.  If $v_2$ is a rich $4$-vertex or $5^+$-vertex, then $v_2\notin Q_4(v)$ and $|Q_4(v)|\leq 2$, so by (R2.2) and (R2.3), $v$ gives at most $\frac{5}{6}$ to  each of $f_1, f_2$, and $v$ gives a charge sequence $(\frac{5}{6}, \frac{5}{6}, \frac{1}{2}, \frac{1}{2}, 1)$ to $f(v)$, it follows that  $\mu^*(v)\geq 4-2\times \frac{5}{6}-2\times \frac{1}{2}-1- 2\times\frac{1}{6}=0$. Therefore,  we may assume that $v_2$ is a poor $4$-vertex. Then  by Lemmas~\ref{l217} (4),  $d(u_1)\geq 5$ as $d(u_2)=3$, and by Lemma~\ref{l220}(4),  and $v_2\not\in Q_4(v)$ or $v_3\not\in Q_4(v)$. Consider $f_5$. By Lemma~\ref{l220}(6), it is either a light $(5,3,4^+, 4^+)$-face, or $(5,3,4,4)$-face with two poor $4$-vertices but $v_5\not\in Q_4(v)$.   By (R2.2) and (R2.3), $v$ gives 1 to $f_2$, $\frac{3}{4}$ to $f_1$, and $\frac{5}{6}$ or 1 to $f_5$ (depend on whether it is light or superlight). Thus,  $\mu^*(v)\geq 4-\frac{3}{4}-1-2\times \frac{1}{2}-\max\{1+\frac{1}{6}, \frac{5}{6}+2\cdot \frac{1}{6}\}=\frac{1}{12}>0$.

\medskip

Assume finally that $M(v)$ contains exactly two $3$-vertices. If $d(u_1)=3$ or $d(u_5)=3$,  then by Lemma~\ref{l220} (3), $M(v)$ contains exactly one $3$-vertex, contrary to our assumption, so by symmetry, $d(u_2)=d(u_3)=3$ or $d(u_2)=d(u_4)=3$.

Let $d(u_2)=d(u_3)=3$.   By Lemma~\ref{l220} (2),  $d(v_4)\geq 5$ as $d(u_2)=3$, and either $d(v_2)\geq 5$ or $d(v_5)\geq 5$ as $d(u_3)=3$.    By Lemma~\ref{l217}(4), $v_3$ is not a poor $4$-vertex as $d(u_2), d(u_3)<4$.  It follows that $|Q_4(v)|\le 1$.   If $d(v_2)=4$, then $d(v_5)\geq 5$ and $f_2$ is a light $(5, 4, 3, 4^+)$-face, so by (R2.2) and (R2.3), $v$ gives a  charge sequence $(1, \frac{5}{6}, \frac{3}{4}, \frac{1}{2}, \frac{3}{4})$ to $f(v)$, and  $\mu^*(v)\geq 4-1-\frac{5}{6}-2\cdot\frac{3}{4}-\frac{1}{2}-\frac{1}{6}=0$;  if $d(v_2)\geq 5$, then $f_2, f_3$ are $(5, 5^+, 3, 4)$-faces,  so by (R2.2) and (R2.3), $v$ gives a charge sequence $(\frac{3}{4}, \frac{3}{4}, \frac{3}{4}, \frac{1}{2}, 1)$ to  $f(v)$, and  $\mu^*(v)\geq 4-1-3\cdot \frac{3}{4}-\frac{1}{2}-\frac{1}{6}>0$.

Let $d(u_2)=d(u_4)=3$.   By Lemma~\ref{l220} (2), $d(v_3)\geq 5$ and $d(v_4)\geq 5$.  If both $v_2, v_5$ are poor $4$-vertices, then by applying Lemma~\ref{l217}(4) to $v_2$ and $v_5$, respectively,  $d(u_1)\geq 5$ and $d(u_5)\geq 5$ as $d(u_2)=d(u_4)= 3$, so by (R2.2) and (R2.3),  $v$ gives a charge sequence $(\frac{3}{4}, \frac{3}{4}, \frac{1}{2}, \frac{3}{4}, \frac{3}{4})$  to $f(v)$, and  $\mu^*(v)\geq 4-4\times\frac{3}{4}-\frac{1}{2}-2\times\frac{1}{6}=\frac{1}{6}>0$.  Then let $v_2$ be a rich $4$-vertex or a $5^+$-vertex. By (R2.2) and (R2.3), $v$ gives a  charge sequence $(\frac{5}{6}, \frac{3}{4}, \frac{1}{2}, \frac{3}{4}, 1)$ to $f(v)$, and $\mu^*(v)\geq 4-1-\frac{5}{6}-2\cdot \frac{3}{4}-\frac{1}{2}-\frac{1}{6}=0$.

\medskip
\n{\bf Case 3.} $N(v)$ has no $3$-vertex.
\medskip

If $M(v)$ has at most one $3$-vertex, then $f(v)$ has at least  four $(5, 4^+, 4^+, 4^+)$-faces, by (R2.2) and (R2.3), $v$ gives  the charge sequence $(1, \frac{1}{2}, \frac{1}{2}, \frac{1}{2}, \frac{1}{2})$ to $f(v)$,  thus  $\mu^*(v)\geq 4-1-4\times\frac{1}{2}-5\times\frac{1}{6}>0$.  Hence by Lemma~\ref{l220} (1), we assume that  $M(v)$ has exactly two $3$-vertices.   By symmetry, we assume that  $d(u_2)=d(u_3)=3$ or  $d(u_2)=d(u_4)=3$. In the former case,  by Lemma~\ref{l217}(4)  $v_3$ is not a poor $4$-vertex, which implies that $v_3\not\in Q_4(v)$, thus  $f_2, f_3$ are light $(5, 4, 3, 4^+)$-faces or $(5, 5^+, 3, 4^+)$-faces. By (R2.2) and (R2.3), $v$ gives  a charge sequence $(\frac{1}{2}, \frac{5}{6}, \frac{5}{6}, \frac{1}{2}, \frac{1}{2})$  to $f(v)$, thus,  $\mu^*(v)\geq 4-2\times \frac{5}{6}-3\times\frac{1}{2}-4\times\frac{1}{6}=\frac{1}{6}>0$.  In the latter case, by Lemma~\ref{l220} (2), $d(v_1)\ge 5$ or $d(v_3)\ge 5$ as $d(u_4)=3$, and $d(v_1)\ge 5$ or $d(v_4)\ge 5$ as $d(u_2)=3$, so $|Q_4(v)|\le 4$.  Note that  $|Q_4(v)|\not=4$, for otherwise, $d(v_1)\geq 5$, $d(v_j)=4$ for $j\in [5]\setminus \{1\}$, and $f_2$ is a $(5, 4, 3, 4)$-face with $v_2, v_3\in Q_4(v)$, a contradiction to Lemma~\ref{l220}(4).   Therefore, by (R2.2) and (R2.3),  $v$ gives a charge sequence $(\frac{1}{2}, 1, \frac{1}{2}, 1, \frac{1}{2})$ to $f(v)$, and we have $\mu^*(v)\geq 4-2\times 1-3\times\frac{1}{2}-3\times\frac{1}{6}=0$.
\end{proof}

\begin{Lemma}\label{le36}
For each $v\in int(C_0)$, $\mu^*(v)\geq 0$.
\end{Lemma}

\begin{proof}
By Lemmas~\ref{le32} and~\ref{le35}, we may assume that $d(v)\geq 6$.  We may further assume that $d(v)=6$, as when $d(v)\ge 7$, $\mu^*(v)\geq \frac{7}{8}\times 7-6=\frac{1}{8}>0$  by~(\ref{eq3}).

If $t_3=0$, then $t_4+t_p\le 6$, so by~(\ref{eq2}),  $\mu^*(v)\geq 6-(t_4+t_p)+\frac{1}{2}t_p\ge 0$. It $t_3=1$, then $t_4\leq 3$, so $\mu^*(v)\geq 6-\frac{9}{4}-t_4-\frac{1}{2}t_p>0$. If $t_3=2$, then by Proposition~\ref{pr21} $t_4\leq 1$, so $\mu^*(v)\geq 6-2\times \frac{9}{4}-1>0$.  Thus we assumer that $t_3=3$.

By (R2.1), $v$ gives at most $\frac{9}{4}$ to a $(6, 4^-, 3)$-face,  $\frac{7}{4}$ to a $(6, 5^+,3)$-face, and $\frac{3}{2}$ to other incident $3$-faces,  thus $\mu^*(v)\geq 6-\frac{9}{4}k_1-2k_2-\frac{7}{4}k_3-\frac{3}{2}k_4$, where $k_1, k_2, k_3, k_4$ are the numbers of $3$-faces that receive $\frac{9}{4}, 2, \frac{7}{4}, \mbox{at most}\, \frac{3}{2}$ from $v$, respectively.  Note that $k_1+k_2+k_3+k_4=3$, and by Lemma~\ref{l212} (5), $v$ is incident with at most two $(6, 4^-, 3)$-faces, thus $k_1+k_2\leq 2$.  Clearly, $\mu^*(v)\ge 9-\frac{9}{4}\cdot 2-\frac{7}{4}=-\frac{1}{4}$, and $\mu^*(v)<0$ only if $k_1=2$ and $k_3=1$, in which case, $v$ is weak, so by (R2c), $v$ should give $\frac{5}{4}$ instead of $\frac{7}{4}$ to the $(6, 5^+,3)$-face, a contradiction.
\end{proof}

\begin{Lemma}\label{le37}
For each $v\in C_0$,  $\mu^*(v)\geq 0$.
\end{Lemma}

\begin{proof}
Let $d(v)= k$. By Proposition~\ref{pr21}, $k\geq 2$.

If $k= 2$, then by (R4), $\mu^*(v)= 2\times 2-6+2=0$. If $k=3$, then  $v$ cannot be incident with faces in $F_3'\cup F_4'$. In this case, $v$ may be incident with a face in $F_3''\cup F_4''$. By (R4) and (R5), $\mu^*(v)\geq \frac{3}{2}-\frac{3}{2}=0$.  Let $k=4$. If $v$ is incident with a $3$-face in $F_3'$, then it is not incident with other $3$- or $4$-faces, thus by (R4) and (R5), $\mu^*(v)\geq 2-3+1=0$;  if $v$ is incident with faces from $F_3''\cup F_4''$, then by  (R4) and (R5),  $\mu^*(v)\geq 2-\frac{3}{2}\cdot 2+1=0$.

Let $k\ge 5$.  The vertex $v$ is incident with at most $\lfloor \frac{k-2}{2}\rfloor$ faces in $F'$.  By (R3), (R4) and (R5), $$\mu^*(v)\ge (2k-6)-3\cdot \lfloor \frac{k-2}{2}\rfloor-\frac{3}{2}\cdot (k-2-2\cdot \lfloor \frac{k-2}{2}\rfloor)=\frac{k}{2}-3.$$
Thus, $\mu^*(v)\ge 0$ if $k\ge 6$. When $k=5$, $v$ gains $\frac{1}{2}$ from $C_0$, so $\mu^*(v)\ge 0$ as well.
\end{proof}

\medskip

Finally, we consider $\mu^*(C_0)$. For $i\in \{2, 3, 4,5\}$, let  $s_i$ be the number of $i$-vertices on $C_0$. Then $|C_0|\geq s_2+s_3+s_4+s_5$.  By (R5),
\begin{align*}
\mu^*(C_0)&\geq |C_0|+6-2s_2-\frac{3}{2}s_3-s_4-\frac{1}{2}s_5\geq |C_0|+6-\frac{3}{2}(s_2+s_3+s_4+s_5)-\frac{1}{2}s_2\\
        &\geq |C_0|+6-\frac{3}{2}|C_0|-\frac{1}{2}s_2=6-\frac{1}{2}(|C_0|+s_2)
\end{align*}
Note that $|C_0|=3$ or $7$.  If $|C_0|=3$ or $s_2\le 5$, then $\mu^*(C_0)\geq 0$. Hence we may assume that  $|C_0|= 7$ and $(s_2, s_3, s_4, s_5)\in \{(6, 1, 0, 0), (7, 0, 0, 0)\}$. If $s_2= 7$,  then $G= C_0$ and it is trivially superextendable.  If $s_2 = 6$ and $s_3 = 1$, then by (R5), $C_0$ gains $1$ from the adjacent face which has  degree more than $7$. Thus, $\mu^*(C_0)\geq \frac{1}{2}>0$.

We have shown that all vertices and faces have non-negative final charges. Furthermore, the outer-face  has positive charges, except when $|C_0| = 7$ and $s_2 = 5$ and $s_3 = 2$  (the two $3$-vertices must be adjacent and  has a common neighbor not on $C_0$), in which there must be a face other than $C_0$  having degree more than $7$.  Thus the face has positive final charge. Therefore, $\sum_{x\in V(G)\cup F(G)}\mu^*(x)>0$, a contradiction.


\small

\begin{thebibliography}{99}

\bibitem{B12}
O. V. Borodin, Colorings of plane graphs: A survey.  {\em Discrete Math.}, {\bf 313} (2013), 517--539.


\bibitem{BG04}
O. V. Borodin and A. N. Glebov, A sufficient condition for planar graphs  to be $3$-colorable, {\em Diskret Anal Issled Oper. } {\bf 10} (2004), 3--11 (in Russian)


\bibitem{BG11}
O. V. Borodin and A. N. Glebov, Planar graphs with neither 5-cycles nor  close 3-cycles are 3-colorable, {\em J. Graph Theory}, {\bf 66} (2011), 1--31.

\bibitem{BGRS05}
O. V. Borodin, A. N. Glebov, A. R. Raspaud, and M. R. Salavatipour.  Planar graphs without cycles of length from 4 to 7 are 3-colorable.  {\em J. of Combin. Theory}, Ser. B, {\bf 93} (2005), 303--311.


\bibitem{BR03}
O. V. Borodin and A. Raspaud, A sufficient condition for planar graphs  to be $3$-colorable, {\em J. Combin. Theory}, Ser B, {\bf 88} (2003), 17--27.




\bibitem{CHMR11}
G. Chang, F. Havet, M. Montassier, and A. Raspaud, Steinberg's  Conjecture and near colorings, manuscript.

\bibitem{DKT09}
Z. Dv\"or\'ak, D. Kr\'al and R. Thomas, Coloring planar graphs  with triangles far apart, arXiv:0911.0885, 2009.

\bibitem{Gar}
Michael R. Garey, David S. Johnson, Computers and Intractability: A Guide  to the Theory of NP--Completeness, W. H. Freeman, ISBN: 0-7167-1045-5, 1979.

\bibitem{G59}
H. Gr\"{o}tzsch,  Ein dreifarbensatz f$\:{u}$r dreikreisfreienetze  auf der kugel. \emph{Math.-Nat.Reihe}, {\bf 8} (1959), 109--120.

\bibitem{H69}
I. Havel, On a conjecture of Grunbaum, {\em J. Combin. Theory Ser. B}, {\bf 7} (1969), 184--186.



\bibitem{HSWXY13}
O. Hill, D. Smith, Y. Wang, L. Xu, and G. Yu, Planar graphs without  $4$-cycles and $5$-cycles are $(3,0,0)$-colorable,  {\em Discrete Math.}, {\bf 313} (2013), 2312--2317.

\bibitem{HY13}
O. Hill, and G. Yu,  A relaxation of Steinberg's Conjecture,  {\em SIAM J. of Discrete Math.},  {\bf 27} (2013), 584--596.

\bibitem{LLY15a}
R. Liu, X. Li, and G. Yu, A relaxation of the Bordeaux Conjecture, {\em European J. of Comb.}, {\bf 49} (2015), 240--249.


\bibitem{LLY15b}
R. Liu, X. Li, and G. Yu, Planar graphs without 5-cycles and intersecting  triangles are $(1, 1, 0)$-colorable, http://arxiv.org/abs/1409.4054 (Submitted).


\bibitem{S76}
R. Steinberg,  The state of the three color problem. Quo Vadis,  Graph Theory?, {\em Ann. Discrete Math.} {\bf 55} (1993), 211--248.

\bibitem{X07}
B. Xu, A $3$-color theorem on plane graph without 5-circuits,  {\em Acta Math Sinica}, {\bf  23}  (2007), 1059--1062.


\bibitem{X08}
B. Xu, On $(3,1)^*$-coloring of planar graphs,  {\em SIAM J. Disceret Math.}, {\bf 23} (2009), 205--220.

\bibitem{XMW12}
L. Xu, Z. Miao, and Y. Wang, Every planar graph with cycles of  length neither $4$ nor $5$ is $(1,1,0)$-colorable, {\em J Comb. Optim.}, {\bf 28} (2014), 774--786.

\bibitem{XW13}
L. Xu  and Y. Wang,  Improper colorability of planar graphs with cycles of length neither 4 nor 6 (in Chinese), {\em Sci .Sin. Math},  {\bf 43} (2013), 15-24.




\end{thebibliography}
\end{document}